\documentclass{article}

\usepackage{amsmath,amsthm,amssymb}
\usepackage{enumerate}
\usepackage{graphicx,subfigure}
\usepackage{cite}
\usepackage{mathtools}
\usepackage{authblk}
\usepackage{tikz}
\usepackage{hyperref}
\theoremstyle{plain} 
\newtheorem{thm}{Theorem}%[section]
\newtheorem{cor}[thm]{Corollary}
\newtheorem{conj}{Conjecture}
\newtheorem{lem}[thm]{Lemma}

\theoremstyle{definition}

\theoremstyle{remark}

\newtheorem{ex}[thm]{Example}
\newtheorem*{notation}{Notation}

\newcommand{\N}{\mathbb{N}}

\pgfdeclarelayer{background}
\pgfdeclarelayer{foreground}
\pgfsetlayers{background,main,foreground}
\tikzstyle{vertex}=[draw,circle,fill=blue!20]
\tikzstyle{small vertex}=[draw,circle,fill=blue!20,inner sep = 0pt,minimum size = .2cm]
\tikzstyle{redvertex} = [draw,vertex, fill=red!40]
\tikzstyle{greenvertex} = [draw,vertex, fill=green!40]
\tikzstyle{labels} = [shape=rectangle, fill = blue!20]
\tikzstyle{edge} = [draw,thick,-]
\tikzstyle{rededge} = [draw,line width=5pt,-,red!40]
\tikzstyle{greenedge} = [draw,line width=5pt,-,green!40]
\tikzset{onslide/.code args={<#1>#2}{%
  \only<#1>{\pgfkeysalso{#2}} % \pgfkeysalso doesn't change the path
}}

\begin{document}
\author[1]{Alex J. Chin}
\author[2]{Gary Gordon}
\author[3]{Kellie J. MacPhee}
\author[2]{Charles Vincent}
\affil[1]{North Carolina State University, Raleigh, NC}
\affil[2]{Lafayette College, Easton, PA}
\affil[3]{Dartmouth College, Hanover, NH}
\title{Random subtrees of complete graphs}
\maketitle

\begin{abstract}  We study the asymptotic behavior of four statistics associated with subtrees of complete graphs:  the uniform probability $p_n$ that a random subtree is a spanning tree of $K_n$, the weighted probability $q_n$ (where the probability a subtree is chosen is proportional to the number of edges in the subtree) that a random subtree spans  and the two  expectations associated with these two probabilities.  We find $p_n$ and $q_n$ both approach $e^{-e^{-1}}\approx .692$, while both expectations approach the size of a spanning tree, i.e., a random subtree of $K_n$ has approximately $n-1$ edges.
\end{abstract}

\section{Introduction}  We are interested in the following two questions:
  \begin{center}
\begin{itemize}
\item [Q1.] What is the asymptotic probability that a random subtree of $K_n$ is a spanning tree?
\item [Q2.] How many edges (asymptotically) does a random subtree of $K_n$ have?
\end{itemize}

\end{center}

In answering both questions, we consider two different probability measures:  a uniform random probability $p_n$, where each subtree has an equal probability of being selected, and a weighted probability $q_n$, where the probability a subtree is selected is proportional to its size (measured by the number of edges in the subtree).  

As expected, weighting subtrees by their size increases the chances of selecting a spanning tree, i.e., $p_n<q_n$.  Table~\ref{globalprobdata} gives data for these values when $ n \leq  100$.
\begin{table}[htdp]
\begin{center}
\begin{tabular}{c|ll}
$n$ & $p_n$ & $q_n$ \\ \hline
 10 & 0.617473 & 0.652736 \\
 20 & 0.657876 & 0.672725 \\
 30 & 0.669904 & 0.679294 \\
 40 & 0.675689 & 0.682552 \\
 50 & 0.67909 & 0.684497 \\
 60 & 0.681329 & 0.685789 \\
 70 & 0.682915 & 0.686711 \\
 80 & 0.684097 & 0.687401 \\
 90 & 0.685012 & 0.687936 \\
 100 & 0.685741 & 0.688365 \\
\end{tabular}
\caption{Probabilities of selecting a spanning tree using uniform and weighted probabilities.}
\label{globalprobdata}
\end{center}
\end{table}

The (somewhat) surprising result is that $p_n$ and $q_n$ approach the same limit as $n \to \infty$.  This is our first main result, completely answering Q1.

\begin{thm}\label{T:main1} 
\begin{enumerate}
\item  Let $p_n$ be the probability of choosing a spanning tree among all subtrees of $K_n$ with uniform probability, i.e., the probability any subtree is selected is the same. Then $$\lim_{n \to \infty} p_n=e^{-e^{-1}}=0.692201\dots$$
\item  Let $q_n$ be the probability of choosing a spanning tree among all subtrees of $K_n$ with weighted probability, i.e., the probability any subtree is selected is proportional to its number of edges. Then $$\lim_{n \to \infty} q_n=e^{-e^{-1}}=0.692201\dots$$
\end{enumerate}
\end{thm}

For the second question Q2, the expected number of edges of a random subtree of $K_n$ is $\sum pr(T) |E(T)|$, where $pr(T)$ is the probability a tree $T$ is selected, $E(T)$ is the edge set of $T$, and the sum is over all subtrees $T$ of $K_n$.  Since we have two distinct probability functions $p_n$ and $q_n$, we obtain two distinct expected values.

The relation between these two expected values is equivalent to a famous example from elementary probability:  

\begin{quote}
All universities  report ``average class size.''  However, this average depends on whether you first choose a class at random, or first select a student at random, and then ask that student to randomly select one of their classes.
\end{quote}

Our uniform expectation is exactly analogous to the first situation, and our weighted expectation is equivalent to the student weighted average.  In this context, edges play the role of students and the subtrees are the classes.  It is a straightforward exercise to show the student weighting always produces a larger expectation.  This was first noticed by  Feld and Grofman in 1977 in \cite{fg}.  For our purposes, this result will show the weighted expectation is always greater than the uniform expectation.

Since both of these expected values are  obviously bounded above by $n-1$, the size of a spanning tree, we  use a variation of {\it subtree density} first defined in \cite{jam1}.  We  divide our expected values by $n-1$, the number of edges in a spanning tree, to convert our expectations to densities.  Letting $a_k$ equal the number of $k$-edge subtrees in $K_n$, this gives us two formulas for subtree density, one  using uniform probability and one  using weighted probability:
$$\mbox{Uniform density: } \mu_p(n)=\frac{\sum_{k=1}^{n-1}ka_k}{(n-1)\sum_{k=1}^{n-1}a_k} \hskip.5in  \mbox{Weighted density: }  \mu_q(n)=\frac{\sum_{k=1}^{n-1}k^2a_k}{(n-1)\sum_{k=1}^{n-1}ka_k}$$

Table~\ref{globaldensities} gives data for these two densities when $n \leq 100$.

\begin{table}[htdp]
\begin{center}
\begin{tabular}{c|ll}
$n$ & $\mu_p(n)$ & $\mu_q(n)$ \\ \hline
 10 & 0.945976 & 0.952436 \\
 20 & 0.977928 & 0.97912 \\
 30 & 0.986177 & 0.986661 \\
 40 & 0.989945 & 0.990205 \\
 50 & 0.9921 & 0.992263 \\
 60 & 0.993496 & 0.993607 \\
 70 & 0.994472 & 0.994553 \\
 80 & 0.995194 & 0.995255 \\
 90 & 0.995749 & 0.995797 \\
 100 & 0.996189 & 0.996228 \\
\end{tabular}
\caption{Subtree densities using uniform and weighted probabilities.}
\label{globaldensities}
\end{center}
\end{table}%

Evidently, both of these densities approach 1.  This is our second main result, and our answer to Q2.

\begin{thm}\label{T:main2}
\begin{enumerate}
\item $\displaystyle{\lim_{n\to\infty}\mu_p(n) = 1,}$
\item $\displaystyle{\lim_{n\to\infty} \mu_q(n) = 1.}$
\end{enumerate}
\end{thm}

The fact that the probabilities and the densities do not depend on which probability measure we use is an indication of the dominance of the number of spanning trees in $K_n$ compared to the the number of non-spanning trees.  Theorems~\ref{T:main1} and \ref{T:main2} are proven in Section~\ref{S:global}.  The proofs follow from a rather detailed analysis of the growth rate of individual terms in the sums that are used to compute all of the statistics.  But we emphasize that these proof techniques are completely elementary.  

Subtree densities have been studied before, but apparently only when the graph is itself a tree.  Jamison introduced this concept in \cite{jam1} and studied its properties in \cite{jam2}.  A more recent paper of Vince and Wang \cite{vw} characterizes extremal families of trees with the largest and smallest subtree densities, answering one of Jamison's questions.  A recent survey of results connecting subtrees of trees with other invariants, including the Weiner index, appears in \cite{sw}.

There are several interesting directions for future research in this area.  We indicate some possible projects in Section~\ref{S:future}.

%We thank several people for useful conversations and communications; Maria Chudnovsky, Evan Fisher, Jeremy Martin,  Robin Pemantle, Derek Smith,  Andrew Vince and Hua Wang.

\section{Global probabilities}\label{S:global}
Our goal in this section is to provide proofs of Theorems~\ref{T:main1} and \ref{T:main2}.    As usual, $K_n$ represents the complete graph on $n$ vertices.  We fix notation for the subtree enumeration we will need.  
\begin{notation}  Assume $n$ is fixed.    We define $a_k, b_k, A$ and $B$ as follows.

\begin{itemize}
\item Let $a_k$ denote the number of $k$-edge subtrees in $K_n$.  (We ignore subtrees of size 0, although setting $a_0=n$ will not change the asymptotic behavior of any of our statistics.)
\item Let $\displaystyle{A=\sum_{k=1}^{n-1}a_k}$ be the total number of subtrees of all sizes in $K_n$.
\item Let $b_k=ka_k$ denote the number of edges used by all of the $k$-edge subtrees in $K_n$.
\item Let $\displaystyle{B=\sum_{k=1}^{n-1}b_k}$ be the sum of the sizes (number of edges) of all the subtrees of $K_n$.
\end{itemize}
\end{notation}

It is immediate from Cayley's formula that  $\displaystyle{ a_k= \binom{n}{k+1} (k+1)^{k-1}.}$  We can view $B$ as the sum of all the entries in a 0--1 edge--tree incidence matrix. 
The four statistics we study here, $p_n, q_n, \mu_p(n)$ and $\mu_q(n)$, can be computed using $A, B, a_k$ and $b_k$.  We omit the straightforward proof of the next result.

\begin{lem}\label{L:global} Let $p_n, q_n, \mu_p(n)$ and $\mu_q(n)$ be as given above.  Then
\begin{enumerate}
\item $\displaystyle{p_n=\frac{a_{n-1}}{A}}$,
\item $\displaystyle{q_n=\frac{b_{n-1}}{B}}$,
\item $\displaystyle{\mu_p(n)=\frac{B}{(n-1)A}}$,
\item $\displaystyle{\mu_q(n)=\frac{\sum_{k=1}^{n-1}kb_k}{(n-1)B}=\frac{\sum_{k=1}^{n-1}k^2a_k}{(n-1)B}}$.
\end{enumerate}
\end{lem}

\begin{ex}  We compute each of these statistics for the graph $K_4$.  In this case, there are 6 subtrees of size one (the 6 edges of $K_4$), 12 subtrees of size two and 16 spanning trees (of size three).  Then we find 
 $\displaystyle{p_4=\frac{16}{34}=.471\dots}$,
$\displaystyle{q_4=\frac{48}{78}=.615\dots}$,
 $\displaystyle{\mu_p(4)=\frac{78}{102}=.768\dots}$,
and $\displaystyle{\mu_q(4)=\frac{198}{234}=.846\dots}$.

\end{ex}

The remainder of this section is devoted to proofs of Theorems~\ref{T:main1} and \ref{T:main2}.  We first prove part (1) of Theorem~\ref{T:main2}, then use the bounds from Lemmas~\ref{lemmatop} and \ref{lemmabottom} to help prove both parts of Theorem~\ref{T:main1}.  Lastly, we  prove part (2) of Theorem~\ref{T:main2}.  

Thus, our immediate goal is to prove that the uniform density $\displaystyle{\mu_p(n)=\frac{B}{(n-1)A}}$ approaches 1 as $n \to \infty$.  Our approach is as follows: We bound the numerator $B$ from below and the term $A$ in the denominator from above so that $\mu_p(n)$ is bounded below by a function that approaches 1 as $n \to \infty$.  Lemma~\ref{lemmatop} establishes the lower bound for $B$ and Lemma~\ref{lemmabottom} establishes the upper bound for $A$, from which the result follows.

%LEMMA TOP
\begin{lem}
\label{lemmatop}  Let  $\displaystyle{B=\sum_{k=1}^{n-1}k\binom{n}{k+1} (k+1)^{k-1}}$, as above.  Then 
\begin{equation}
B > (n-1)n^{n-2} \left( \frac{n-3}{n-1}\right) e^{e^{-1}}.
\label{topresult}
\end{equation}

\end{lem}

%We note that we use $n-3$ for simplicity.  Any $n-k$ such that $k \geq 1+e^{\frac{1}{e}} \approx 2.4447$ suffices.

\begin{proof}
Recall  $b_{n-1} = (n-1)n^{n-2}$ counts the total number of edges used in all the spanning trees of $K_n$.  We examine the ratio $b_{i}/b_{i-1}$ in order to establish a lower bound for each $b_i$ in terms of $b_{n-1}$.  
\begin{equation}
\frac{b_i}{b_{i-1}} = \frac{\binom{n}{i+1} (i+1)^{i-1} i}{\binom{n}{i} i^{i-2} (i-1)} = \frac{n-i}{i+1}\cdot \frac{i}{i-1}\cdot \frac{(i+1)^{i-1}}{i^{i-2}} = (n-i)\cdot\frac{i}{i-1} \cdot\left(\frac{i+1}{i}\right)^{i-2}
\label{topratio}
\end{equation}
We use the fact that $e$ is the least upper bound for the sequence $\left\{\left(\frac{i+1}{i}\right)^{i-2}\right\}$ to rewrite \eqref{topratio} as the inequality
$$
b_{i-1} \geq \frac{i-1}{i(n-i)e} \cdot b_i,
$$
and this is valid for  $i = 2, 3,\dots,n-1$.  In general, for $i<n$, an inductive argument on $n-i$  establishes the following:
\begin{equation}
b_{i} \geq \frac{i}{(n-i-1)!(n-1)e^{n-i-1}} \cdot b_{n-1}.
\label{topinequality}
\end{equation}
%
%When $i = n-1$, we have
%\begin{equation}
%t_{n-2} \geq \frac{n-2}{(n-1)e} \cdot t_{n-1}.
%\label{topn-1}
%\end{equation}
%When $i = n-2$,
%\begin{equation}
%t_{n-3} \geq \frac{n-3}{2(n-2)e} \cdot t_{n-2} \geq \frac{n-3}{2(n-1)e^2} \cdot t_{n-1},
%\end{equation}
%where we have used \eqref{topn-1} to substitute in for $t_{n-1}$.  Continuing in this manner, we obtain
%\begin{equation}
%t_{n-4} \geq \frac{n-4}{6(n-1)e^3} \cdot t_{n-1}
%\end{equation}
%for $i = n-4$, and, in general,

%
Then equation \eqref{topinequality} gives
\begin{equation}
B = \sum_{i=1}^{n-1} b_i \geq \frac{b_{n-1}}{n-1} \left((n-1) + \frac{n-2}{e} + \frac{n-3}{2e^2} + \frac{n-4}{6e^3} + \dots + \frac{1}{(n-2)!e^{n-2}}\right).
\end{equation}
We bound this sum below using standard techniques from calculus.  Let  
$$h(k) = 1 + \frac{1}{e} + \frac{1}{2e^2} + \frac{1}{6e^3} + \dots + \frac{1}{k!e^{k}}.$$

Then
\begin{equation}
B \geq \frac{b_{n-1}}{n-1}(h(n-2) + h(n-3) + \dots + h(0)).
\label{deriv}
\end{equation}
Now  $\displaystyle{e^{x} = \sum_{i=0}^{k}\frac{x^i}{i!} + R_k(x)}$, where $\displaystyle{R_k(x) = \frac{e^y}{(k+1)!} x^{k+1}}$
for some  $y \in (0,x)$.  We are interested in this expression when $x=e^{-1}$.  Then $R_k(x)$ is maximized when $x=y= e^{-1}$.  So, using $e^{(e^{-1})} \approx 1.44 \ldots<2$, we have
$$
R_k\left(e^{-1}\right) \leq  \frac{e^{(e^{-1})}}{{(k+1)!} }(e^{-1})^{k+1} \leq \frac{2}{(k+1)!e^{k+1}}.
$$
Now, using this upper bound on the error in the Maclaurin polynomial for $e^x$ at  $x = e^{-1}$ gives
$$
e^{e^{-1}} = h(k) + R_k\left(e^{-1}\right) \leq h(k) + \frac{2}{(k+1)!e^{k+1}},
$$
so 
$$
h(k) \geq e^{e^{-1}} - \frac{2}{(k+1)!e^{k+1}}.
$$
Substituting into \eqref{deriv},
\begin{align*}
B &\geq \frac{b_{n-1}}{n-1}\sum_{k=0}^{n-2} h(k)
\geq \frac{b_{n-1}}{n-1} \left((n-1)e^{e^{-1}} - 2\sum_{k=0}^{n-2}\frac{1}{(k+1)!e^{k+1}}\right) \\
&> \frac{b_{n-1}}{n-1} \left((n-1)e^{e^{-1}} - 2\sum_{i=0}^\infty\frac{1}{i!e^i}\right)
= \frac{b_{n-1}}{n-1} \left((n-1)e^{e^{-1}} - 2e^{e^{-1}}\right) = b_{n-1} \left( \frac{n-3}{n-1}\right)e^{e^{-1}}.
\end{align*}

\end{proof}
%
%
%
%LEMMA BOTTOM

We now give an upper bound for $A$, the total number of subtrees of $K_n$.
\begin{lem}
\label{lemmabottom}
Let $\displaystyle{A=\sum_{k=1}^{n-1}a_k}$ be the total number of subtrees of all sizes in $K_n$, as above.   Then, for every $\varepsilon >0$, there is a positive integer $r(\varepsilon) \in \N$ so that, for all $n>r(\varepsilon)$, 
\begin{equation}
A <n^{n-2} \left(e^{(e-\varepsilon)^{-1}} + \frac{e}{r(\varepsilon)!}\right)
\label{bottomresult}
\end{equation}
\end{lem}

\begin{proof}
Recall $a_{n-1}=n^{n-2}$ is the number of spanning trees in $K_n$.  As in Lemma~\ref{lemmatop}, we examine ratios of consecutive terms, but this time we need to establish an upper bound for the $a_i$ in terms of $a_{n-1}$.  Now
\begin{equation}
\frac{a_i}{a_{i-1}} = \frac{\binom{n}{i+1} (i+1)^{i-1}}{\binom{n}{i} i^{i-2}} = \frac{n-i}{i+1}\cdot \frac{(i+1)^{i-1}}{i^{i-2}} = (n-i) \left(\frac{i+1}{i}\right)^{i-2}
\label{bottomratio}
\end{equation}
Let $\varepsilon > 0$.  Since $\lim_{i\to\infty} \left(\frac{i+1}{i}\right)^{i-2} = e$, there exists a $k(\varepsilon)$ such that
$$
a_{i-1} \leq \frac{a_i}{(n-i)(e-\varepsilon)}
$$
for every $k(\varepsilon) < i \leq n-1$.
%%
%When $i = n-1$,
%\begin{equation}
%b_{n-2} \leq \frac{1}{e-\varepsilon} \cdot b_{n-1}.
%\label{bottomn-1}
%\end{equation}
%When $i = n-2$,
%\begin{equation}
%b_{n-3} \leq \frac{1}{2(e-\varepsilon)} \cdot b_{n-2} \leq \frac{1}{2(e-\varepsilon)^2} \cdot b_{n-1},
%\end{equation}
%where we have used \eqref{bottomn-1} to substitute in for $b_{n-1}$.  Continuing in this manner, we obtain
%\begin{equation}
%b_{n-4} \leq \frac{1}{6(e-\varepsilon)^3} \cdot b_{n-1}
%\end{equation}
%for $i = n-4$, and, in general,
As in the proof of Lemma~\ref{lemmatop}, an inductive argument can be used to show 
$$
a_i \leq \frac{a_{n-1}}{(n-i-1)!(e-\varepsilon)^{n-i-1}}
$$
for all $i$ such that $k(\varepsilon) < i \leq n-1$.
On the other hand, if $i\leq k(\varepsilon)$, then
$$
a_{i-1}= \frac{a_i}{(n-i)\left(\frac{i+1}{i}\right)^{i-2}}  \leq \frac{a_i}{n-i} \leq \frac{a_{n-1}}{(n-i)!}
$$
where we have bounded $\left(\frac{i+1}{i}\right)^{i-2}$ below by 1 and the final inequality follows by a similar inductive argument.
Therefore, 
\begin{equation*}
A = \sum_{i=1}^{n-1}a_i \leq a_{n-1} (f(n,k(\varepsilon)) + g(n,k(\varepsilon)))
\end{equation*}
where
\begin{equation*}
f(n,k(\varepsilon)) = 1 + \frac{1}{e-\varepsilon} + \frac{1}{2!(e-\varepsilon)^2} + \dots + \frac{1}{(n-k(\varepsilon))!(e-\varepsilon)^{n-k(\varepsilon)}},
\end{equation*}
corresponding to those terms where $i>k(\varepsilon)$, and
\begin{equation*}
g(n,k(\varepsilon)) = \frac{1}{(n-k(\varepsilon)+1)!} + \frac{1}{(n-k(\varepsilon)+2)!} + \dots + \frac{1}{(n-2)!} + \frac{1}{(n-1)!}
\end{equation*}
corresponds to the terms where $i\leq k(\varepsilon)$.
Using the Maclaurin expansion for $e^x$ evaluated at $x = (e-\varepsilon)^{-1}$ gives an upper bound for $f(n,k(\varepsilon))$:
\begin{equation*}
f(n,k(\varepsilon)) = \sum_{i=0}^{n-k(\varepsilon)} \frac{1}{i!(e-\varepsilon)^i} < \sum_{i=0}^{\infty} \frac{1}{i!(e-\varepsilon)^i} = e^{(e-\varepsilon)^{-1}}.
\end{equation*}
For $g(n,k(\varepsilon))$, we have
\begin{equation*}
g(n,k(\varepsilon)) = \sum_{i = n-k(\varepsilon)+2}^{n} \frac{1}{(i-1)!} < \sum_{i=n-k(\varepsilon)+2}^{\infty} \frac{1}{(i-1)!}  <\frac{e}{r(\varepsilon)!},
\end{equation*}
where $r(\varepsilon) = n-k(\varepsilon) + 1$.
Therefore,
\begin{equation*}
A \leq a_{n-1}(f(n,k(\varepsilon)) + g(n,k(\varepsilon))) < a_{n-1}\left(e^{(e-\varepsilon)^{-1}} + \frac{e}{r(\varepsilon)!}\right).
\end{equation*}

\end{proof}

We can now prove part (1) of Theorem~\ref{T:main2}.

\begin{proof} [Proof: Theorem~\ref{T:main2} (1)]
Recall   $b_{n-1} = (n-1)n^{n-2}$ and $a_{n-1} = n^{n-2}$, so
\begin{equation*}
\frac{1}{n-1} \cdot \frac{b_{n-1}}{a_{n-1}} = 1.
\end{equation*}
Therefore, \eqref{topresult} and \eqref{bottomresult} imply
\begin{equation*}
\mu_p(n) = \frac{1}{n-1} \cdot \frac{B}{A} > \frac{1}{n-1} \cdot \frac{b_{n-1}}{a_{n-1}} \cdot \frac{\left( \frac{n-3}{n-1} \right) e^{e^{-1}} }{e^{(e-\varepsilon)^{-1}} + \frac{e}{r(\varepsilon)!}} =\left( \frac{n-3}{n-1}\right) \cdot \frac{e^{e^{-1}} }{e^{(e-\varepsilon)^{-1}} + \frac{e}{r(\varepsilon)!}}.
\end{equation*}
Then
$\displaystyle{\lim_{n\to\infty} \mu_p(n) = 1}$ as $n \to \infty$ since  $\varepsilon$ can be chosen arbitrarily small and $r(\varepsilon)$ can be made arbitrarily large.

\end{proof}

We now prove Theorem~\ref{T:main1}.

\begin{proof}[Proof: Theorem~\ref{T:main1}]
\begin{enumerate}
\item Recall $p_n=\frac{a_{n-1}}{A}$, where $a_{n-1}$ is the number of spanning trees in $K_n$ and $A$ is the total number of subtrees of $K_n$.  Then the argument in Lemma~\ref{lemmabottom} can be modified to prove 
$$A \geq a_{n-1}\sum_{i=0}^{n-1}\frac{1}{i!e^i}.$$
This follows by bounding $\displaystyle{ \left(\frac{i+1}{i}\right)^{i-2}}$ below by $e$ in equation~\eqref{bottomratio} -- this is the same bound we needed in our proof of Lemma~\ref{lemmatop}.  Then, bounding $\displaystyle{ \frac{1}{p_n}=\frac{A}{a_{n-1}} }$, we have

$$e^{e^{-1}}-\varepsilon' \leq \sum_{1=0}^{n-1}\frac{1}{i!e^i} \leq \frac{A} {a_{n-1}}\leq \left(e^{(e-\varepsilon)^{-1}} + \frac{e}{r(\varepsilon)!}\right),$$ where $\varepsilon$ and $\varepsilon'$ can be made as small as we like, and $r(\varepsilon)$ can be made arbitrarily large.

Hence $$p_n=\frac{a_{n-1}}{A} \to \frac{1}{e^{e^{-1}}} \approx 0.6922\dots$$ as $n \to \infty$.

\item Note that $$q_n=\frac{(n-1)a_{n-1}}{B}=\frac{(n-1)A}{B}\cdot \frac{a_{n-1}}{A} = \frac{p_n}{\mu_p(n)}.$$

By Theorem~\ref{T:main2}(1), $\mu_p(n)\to 1$ and, by part (1) of this theorem, $p_n \to e^{-e^{-1}}$.  The result now follows immediately.

\end{enumerate}
\end{proof}

An interesting consequence of our proof of part (2) of Theorem~\ref{T:main1} is that $p=q\mu_p(n)$ for any graph $G$, so $p<q$.  Thus, if $G=C_n$ is a cycle, then the (uniform) density is approximately $\frac12$, so we immediately get the weighted probability that a random subtree spans is approximately twice the probability for the uniform case (although both probabilities approach 0 as $n \to \infty$).  

We state this observation as a corollary.

\begin{cor}\label{C:pnqn}  Let $G$ be any connected graph, let $p(G)$ and $q(G)$ be the uniform and weighted probabilities (resp.) that a random subtree is spanning, and let $\mu(G)$ be the uniform subtree density.  Then $p(G)=q(G)\mu(G).$

\end{cor}

If $\mathcal{G}$ is an infinite family of graphs, then we can interpret Cor.~\ref{C:pnqn} asymptotically.  In this case, the two probabilities $p$ and $q$ coincide (and are non-zero) in the limit if and only if the density approaches 1.

We conclude this section with a very short proof of  part (2) of Theorem~\ref{T:main2}. 

\begin{proof}[Proof: Theorem~\ref{T:main2} (2)]  We have  $\mu_p(n)<\mu_q(n)<1$ for all $n$ by a standard argument in probability (see \cite{fg}). Since $\mu_p(n) \to 1$ as $n \to \infty$, we are done.

\end{proof}

%
%We close this section with a strange application of part (1) of Theorem~\ref{T:main2}.  Let $\overline{u}_n, \overline{v}$ and $ \overline{w}$ new the following vectors in $\mathbb{R}^n$.
%$$\overline{u}_n=\langle1,1, \dots, 1\rangle \hskip.3in \overline{v}_n=\langle0,1,2, \dots, n-1\rangle \hskip.3in 
%\overline{w}_n=\langle a_0,a_1, \dots, a_{n-1}\rangle, $$
%where the $a_i$ are the number of subtrees of $K_n$ using $i$ edges.
\section{Conjectures, extensions and open problems}\label{S:future}
We believe the study of subtrees in arbitrary graphs is a fertile area for interesting research questions.  We outline several ideas that should be worthy of future study.  Many of these topics are addressed in \cite{cgmv}.

\begin{enumerate}
\item {\bf Local statistics.}  Compute ``local'' versions of the four statistics given here.  Fix an edge $e$ in $K_n$.  Then it is straightforward to compute $p_n', q_n', \mu_{p'}(n)$ and $\mu_{q'}(n)$, where each of these is described below. 
\begin{itemize}
\item $p_n'$ is the (uniform) probability  that a random subtree {\it containing the edge $e$} is a spanning tree.  This is given by
$$p_n'=\frac{n^{n-3}}{\sum_{k=0}^{n-2}{n-2 \choose k}(k+2)^{k-1}}.$$
The number of spanning trees containing a given edge is $2n^{n-3}$ -- this is easy to show using the tree-edge incidence matrix.  Incidence counts can also show the total number of subtrees containing the edge $e$ is given by $\displaystyle{\sum_{k=0}^{n-2}2{n-2 \choose k}(k+2)^{k-1}}$.  This can also be derived by using the {\it hyperbinomial transform} of the sequence of 1's.  

\item $q_n'$ is the  (weighted) probability that a random subtree {\it containing the edge $e$} is a spanning tree.  This time, we get
$$q_n'=\frac{(n-1)n^{n-3}}{\sum_{k=0}^{n-2}{n-2 \choose k}(k+2)^{k-1}}.$$

\item $ \mu_{p'}(n)$ is the local (uniform) density, so 
$$\mu_{p'}(n)=\frac{\sum_{k=0}^{n-2}(k+1){n-2 \choose k}(k+2)^{k-1}}{(n-1)\sum_{k=0}^{n-2}{n-2 \choose k}(k+2)^{k-1}}.$$
\item $ \mu_{q'}(n)$ is the local (weighted) density, which gives
$$\mu_{q'}(n)=\frac{\sum_{k=0}^{n-2}(k+1)^2{n-2 \choose k}(k+2)^{k-1}}{(n-1)\sum_{k=0}^{n-2}(k+1){n-2 \choose k}(k+2)^{k-1}}.$$
\end{itemize}
 It is not difficult to prove the same limits that hold for the global versions of these statistics also hold for the local versions: $p'_n$ and $q'_n$ both approach $e^{-e^{-1}}$ and $ \mu_{p'}(n)$ and $\mu_{q'}(n)$ both approach 1 as $n \to \infty$.  (All of these can be proven by arguments analogous to the global versions.)
 
 For $p'_n$, however, the connection to the global statistics is even stronger:  $p_n'=q_n$ for all $n$, i.e., the global weighted probability exactly matches the local uniform probability of selecting a spanning tree.    (This can be proven by a direct calculation, but it is immediate from the ``average class size'' formulation of the weighted probability.)

Other statistics that might be of interest here include larger local versions of these four:  suppose we consider only subtrees that contain a given pair of adjacent edges, or a given subtree with 3 edges.  Will the limiting probabilities match the ones given here?

\item {\bf Non-spanning subtrees.}  Explore the probabilities that a random subtree of $K_n$ has exactly $k$ edges for specified $k<n-1$.  The analysis given here can be used to show the uniform and weighted probabilities of choosing a subtree with $n-2$ edges (one less than a spanning tree) approaches $\displaystyle{\left(e^{-1-e^{-1}}\right) =0.254646\dots.}$  This follows by observing the ratio $a_{n-1}/a_{n-2} \to e$ as $n \to \infty$.  (It is interesting that both sequences are decreasing here, in contrast to the sequences $p_n$ and $q_n$.)

\item {\bf Other graphs.} Explore these statistics for other classes of graphs.  For instance, when $G=K_{n,n}$ is a complete bipartite graph, we get limiting values for the probability (uniform or weighted) that a random subtree spans is
$$e^{-2e^{-1}} = 0.478965\dots,$$
 the square of the limiting value we obtained for the complete graph.  Both the uniform and weighted densities approach 1 in this case (forced by Cor.~\ref{C:pnqn}).

On the other hand, consider the theta graph $\theta_{n,n}$, formed adding an edge ``in the middle'' of a $2n-2$ cycle (see Fig.~\ref{F:theta}).  Then the uniform density approaches $\frac23$ as $n \to \infty$, and both the uniform and weighted probabilities that a random subtree spans tend to 0 as $n$ tends to $\infty$  \cite{cgmv}.

    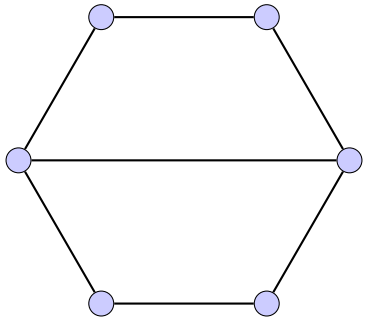
\begin{figure}[h] 
    \centering
        \begin{tikzpicture}[scale=2.2, auto,swap]
            % Draw the vertices
            \foreach \pos/\name in {{(-0.5,0.866)/a}, {(0.5,0.866)/b}, {(1,0)/c}, {(0.5,-0.866)/d}, {(-0.5,-0.866)/e}, {(-1,0)/f}}
                \node[vertex] (\name) at \pos {};
                
            % Labels
%            \node[labels] (labela) at (0,0.4) {$a=4$};
%            \node[labels] (labelb) at (0,-0.4) {$b=4$};
 
            % Connect vertices with edges
            \foreach \source/ \dest in {a/b, b/c, c/d, d/e, e/f, f/a, f/c}
                \path[edge] (\source) edge (\dest);
            
        \end{tikzpicture}
        \caption{$\theta_{4,4}$.}
        \label{F:theta}
    \end{figure}

We conjecture that, for any given graph, the number of edges determines whether or not these limiting densities are 0.

\begin{conj}
If $|E|=O(n^2)$, then the probability (either uniform or weighted)  of selecting a spanning tree is non-zero (in the limit).
\end{conj}

\begin{conj}
If $|E|=O(n)$, then the probability (either uniform or weighted) of selecting a spanning tree is zero (in the limit).
\end{conj}

\item {\bf Optimal graphs.}  Determine the ``best'' graph on $n$ vertices and $m$ edges.  There are many possible interpretations for ``best'' here:  for instance, we might investigate which graph maximizes $p(G)$ and which maximizes the density?  Must the graph that maximizes one of these statistics maximize all of them?  This is closely related to the work in \cite{vw}, where extremal classes of trees are determined for  uniform density.

\item {\bf Subtree polynomial.}  We can define a polynomial to keep track of the number of subtrees of size $k$.  If $G$ is any graph with $n$ vertices, let $a_k$ be the number of subtrees of size $k$.  Then define a {\it subtree polynomial} by 
$$s_G(x) = \sum_{k=0}^{n-1} a_k x^k.$$

We can compute the subtree densities  directly from this polynomial and its derivatives:

$$\mu_p(G) = \frac{s_G'(1)}{(n-1)s_G(1)}  \hskip.2in \mbox{and} \hskip.2in \mu_q(G) = \frac{s_G'(1)+s''_G(1)}{(n-1)s'_G(1)}$$

It would be worthwhile to study the roots and various properties of this polynomial.  In particular, we conjecture that the coefficients of the polynomial are unimodal.

\begin{conj}
The coefficients of $s_G(x)$ are unimodal.
\end{conj}

As an example of an interesting infinite 2-parameter family of graphs, the coefficients of the $\theta$-graph $s_{\theta_{a,b}}(x)$ are unimodal (see Fig.~\ref{F:thetacoef}).
\begin{figure}[h] 
	\centering
	\includegraphics[width=.5\textwidth]{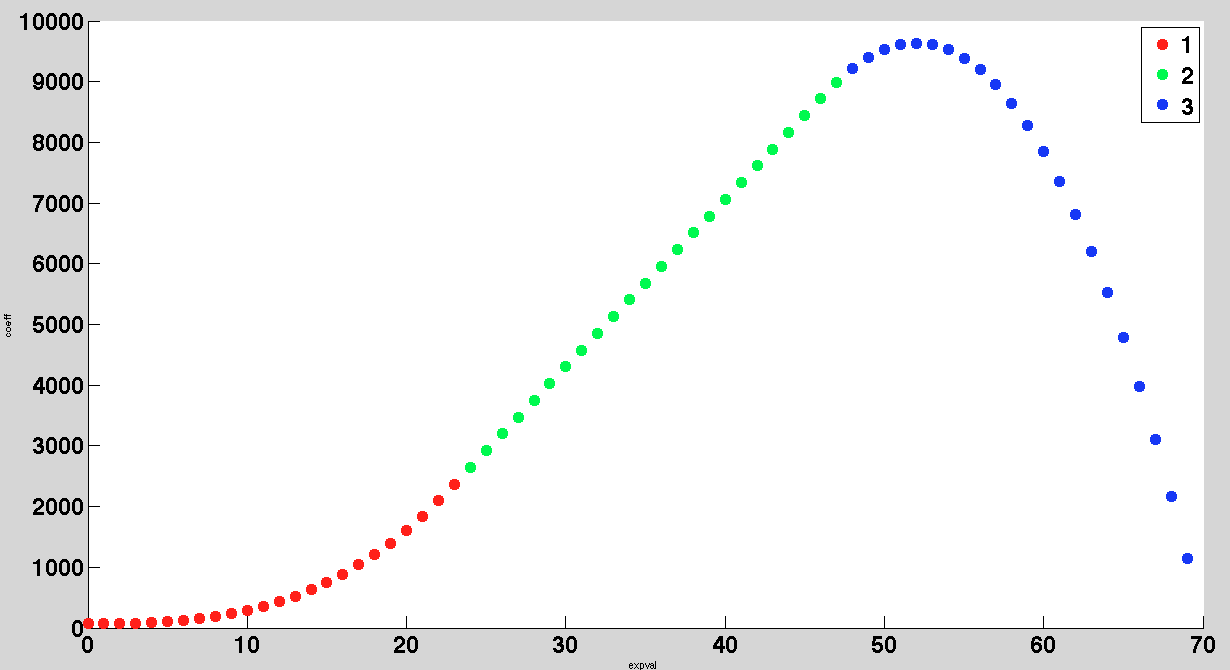}
	\caption{Coefficients of $\theta_{24,48}$, with  three coefficient regions highlighted}
	\label{F:thetacoef}
\end{figure}

The mode of the sequence of coefficients gives another measure of the subtree density.  For the $\theta$-graph $\theta_{n,n}$, we can show the mode is approximately $\sqrt{2}n$ (in \cite{cgmv}).  

Stating the unmodality conjecture in terms of a polynomial has the advantage of potentially using the very well developed  theory of polynomials.  Surveys that address unmodality of coefficients in polynomials include Stanley's classic paper \cite{st} and a recent paper of Pemantle \cite{pe}.  In particular, if all the roots of $s_G(x)$ are negative reals, then the coefficients are unimodal.  (Such polynomials are called {\it stable}.)  Although $s_G(x)$ is not stable in general, perhaps it is for some large class of graphs.

It is not difficult to find different graphs (in fact, different trees)  with the same subtree polynomial.

\begin{conj}
For any $n$, construct $n$-pairwise non-isomorphic graphs that all share the same subtree polynomial.
\end{conj}

This is expected by pigeonhole considerations -- there should be far more graphs on $n$ vertices than potential polynomials.  

\item {\bf Density monotonicity.}  Does adding an edge always increase the density?   This is certainly false for disconnected graphs -- simply add an edge to a small component; this will lower the overall density.  But we conjecture this cannot happen if $G$ is connected:

\begin{conj}
Suppose $G$ is a connected graph, and $G+e$ is obtained from $G$ by adding an edge between two distinct vertices of $G$.  Then $\mu(G)<\mu(G+e)$.
\end{conj}

One consequence of this conjecture would be that, starting with a tree, we can add edges one at a time to create a complete graph, increasing the density at each stage.  This could be a useful tool in studying optimal families.

\item {\bf Matroid generalizations.} Instead of using subtrees of $G$, we could use {\it subforests}.  This has the advantage of being well behaved under deletion and contraction.  In particular, the total number of subforests of a graph $G$ is an evaluation of its Tutte polynomial:  $T_G(2,1)$.  The number of spanning trees is also an evaluation of the Tutte polynomial:  $T_G(1,1)$  (In fact, this is how Tutte defined his {\it dichromatic} polynomial originally.)

All of the statistics studied here would then have direct analogues to the subtree problem.  This entire approach would then generalize to matroids.  In this context, subforests correspond to independent sets in the matroid, and spanning trees correspond to bases.  It would be of interest to study basis probabilities and densities for the class of binary matroids, for example.

\end{enumerate}

%The sum of the number of edges in all of the subtrees of $K_n$ is
%$$\sum _{k=0}^{n-1} k (k+1)^{k-1} \binom{n}{k+1}=(n-1) n \sum _{k=0}^{n-2} (k+2)^{k-1} \binom{n-2}{k}$$
%The LHS comes from counting the contributions of trees of size $k$, and the RHS is from the edge-tree incidence matrix.

\end{document}